\documentclass[letterpaper, 10 pt, conference]{ieeeconf}  

\IEEEoverridecommandlockouts                              

\usepackage{tikz}
\usepackage{amsmath}
\usepackage{amssymb}
\usepackage{dsfont}
\usepackage{tikz}
\usepackage[titlenumbered,ruled]{algorithm2e}
\usepackage{subcaption}
\usepackage{outlines}
\usepackage{balance}

\newtheorem{definition}{Definition}
\newtheorem{theorem}{Theorem}

\newtheorem{lemma}{Lemma}

\newcommand{\Lfi}{\ensuremath{\textbf{L}_{ff}^{-1}}}

\newcommand{\Ne}{\ensuremath{\mathcal N}}

\newcommand{\xf}{\ensuremath{\textbf{x}_f}}
\newcommand{\xl}{\ensuremath{\textbf{x}_l}}

\newcommand{\sz}{\ensuremath{S_0}}
\newcommand{\so}{\ensuremath{S_1}}

\newcommand{\szf}{\ensuremath{\overline{l}_0}}

\newcommand{\Simpson}{{\mathcal{O}_{\text{Sim}}}}
\newcommand{\Shannon}{{\mathcal{O}_{\text{Shan}}}}

\title{\LARGE \bf
Maximizing Diversity of Opinion in Social Networks
}

\author{Erika Mackin and Stacy Patterson
\thanks{Erika Mackin and Stacy Patterson are with the Department of Computer Science, Rensselaer Polytechnic Institute, Troy, New York 12180, USA. Email: {\tt\small mackie2@rpi.edu, sep@cs.rpi.edu}}%
\thanks{This work was funded in part by NSF awards CNS-1527287 and CNS-1553340.}
}

\begin{document}

\maketitle
\thispagestyle{empty}
\pagestyle{empty}

\begin{abstract}
We study the problem of maximizing opinion diversity in a social network that  includes opinion leaders with binary opposing opinions. The members of the network who are not leaders form their opinions using the French-DeGroot model of opinion dynamics. To quantify the diversity of such a system, we adapt two diversity measures from ecology to our setting, the Simpson Diversity Index and the Shannon Index. Using these two  measures, we formalize the problem of how to place a single leader with opinion 1, given a network with a leader with opinion 0, so as to maximize the opinion diversity. We give analytical solutions to these problems for paths, cycles, and trees, and we highlight our results through a numerical example.
\end{abstract}

\section{Introduction}
As social networks and social media become an increasingly important part of our lives, the study of how opinions form and spread in these networks has become a rich area of study. Problems of interest include how quickly members of a social network converge to agreement~\cite{GS14},  how to identify the most influential users~\cite{KKT03}, the effect of friendly vs. antagonistic interactions between users~\cite{ZZ18}, which users to advertise to in order to increase the popularity of a product~\cite{TTZ18},  and so on.

One potential downside of online social media spaces is that they encourage rapid consensus and polarization of opinions~\cite{DGM14, W15}. Although agreement in a population has its benefits, diversity of opinion is also important. A community with diverse opinions is better able to innovate due to a wider variety of potential perspectives \cite{P07}. Further, it has been argued that there are four traits necessary for a crowd to be wise, one of which is diversity of opinion~\cite{S05}.
%

We study a social network where members, or \emph{nodes}, exchange opinions using the French-Degroot model, wherein nodes update their opinion based on both their current opinion and the opinions of their peers \cite{PT18}.
We assume the system contains some nodes who contribute their opinions to their neighbors but never change their own opinions. In real-world social networks, such a person could be a paid promoter of some product or political stance (i.e., social media ``influencers''), or simply very opinionated. These nodes are the opinion leaders in the network, which we refer to simply as \emph{leaders}. 
 We consider a setting in which leaders each have an  opinion of  $0$ or $1$.  For example, in an election, perhaps an opinion of $0$ corresponds to an unwavering decision to vote for Party A's candidate, while $1$ indicates the same commitment to voting for Party B's candidate, and opinions that lie in the interval $(0,1)$ represent varying levels of indecision, with an opinion of $0.5$ indicating a truly undecided voter. 
 In this setting, the follower nodes' opinions converge to values in the interval $[0,1]$. Our aim is to quantify the diversity of these opinions; intuitively, a diverse network has opinions that cover the full spectrum of the opinion interval $[0,1]$.

To formalize this notion of diversity, we  propose two diversity measures: the \emph{Simpson Opinion Diversity Index} and the \emph{Shannon Opinion Diversity Index}.  
The first is based on the Simpson Diversity Index, a concept originally from ecology that measures the diversity of a community composed of many different species \cite{S49}. The Simpson Diversity Index was designed to be maximized when all species in the area are represented equally. 
Our second measure is derived from the Shannon Index, which was originally developed in a communications context to express the probability of receiving a given text string over a communication channel. In this context, it is now generally referred to as information entropy. The Shannon Index is also used in ecology, where the probabilities represent the likelihood of a randomly selected individual belonging to a given species~\cite{S01}. The Shannon Index is maximized when all outcomes are equally probable. 

Our proposed performance measures are parameterized by a number of opinion bins. This can range from two bins,
where a maximally  diverse network is one in which half the population has an opinion in the interval $[0,0.5)$ and the remaining population has an opinion in the interval $[0.5,1]$,
to $n_f$ bins, in a network with $n_f$ follower nodes, in which case a maximally diverse network is one in which the opinions are uniformly distributed over $[0,1]$.


We pose the problem of optimizing  the diversity of the network by selecting which nodes should act as leaders.
Specifically, we consider a network with a single leader node with opinion 0. As is, such a network will converge to a state where all opinions are 0.   We seek to identify a node that, if it becomes  a 1-valued leader,  will maximize the  diversity of the resulting opinions in the network. For instance, if we want to prevent political discussion in an online forum from being dominated by a supporter of Party A, we can invite a strong supporter of Party B into the community as well to encourage a wider variety of opinions.  We present analytical solutions to this problem for both two bins and $n_f$ bins in path, cycle, and tree graphs for both performance measures. Further, we present numerical results highlighting the difference between these performance measures.




%

\subsubsection*{Related Work}
Previous work on controlling the opinions of the network has generally focused on maximizing either the sum or the average of the node states, rather than promoting diverse opinions. Targeted placement of leaders to maximize the followers' opinions has been studied in systems with stubborn agents \cite{HZ18}, agents with both fixed internal and  modifiable external opinions \cite{GTT13}, agents whose stubbornness increases over time~\cite{AKPT18}, and opposing leaders with similar dynamics to our setting \cite{VFFO14}.
In \cite{YOASS13}, they consider a network that consists of binary opposing leaders and followers who update their state via binary voting and study the problem of maximizing the sum of of the node states. And in \cite{MA18}, the problem of how to place a single leader in a directed graph so as to maximize its influence, given the presence of up to two opposing leaders,  is considered.
To the best of our knowledge, no other works have studied the problem of using opposing leaders to encourage diversity of node states.

The remainder of this paper is organized as follows. In  Section \ref{model.sec}, we give the system model, performance measures, and problem formulations.
 We follow this with an analysis of the optimal 1-leader placement in several network topologies in Section \ref{results.sec}. 
 In Section~\ref{example.sec}, we highlight the difference between the diversity measures via a numerical example.
 Finally, we conclude in Section \ref{concl.sec}.

\section{System Model} \label{model.sec}
We consider a set of $n$ individuals, or nodes, making up a connected, undirected, unweighted graph $G=(V,E)$,
with $V$ the set of nodes and $E$ the set of edges.  An edge  $(u,v) \in E$ denotes a social link (friendship, colleagues, etc.) between nodes $u$ and $v$.
The nodes are divided into a set of leaders and a set of followers. The set of leaders is further divided into a set of leaders with  opinion $1$, denoted by $S_1$, and a set of leaders with opinion $0$, denoted by $S_0$.
We call these sets the 1-\emph{leader set} and the 0-\emph{leader set}, respectively. 
The set of followers is $F = V \setminus (S_0 \cup S_1)$.
We let $n_f = |F|$ and $n_l = | S_1 \cup S_0|$, so that $n = n_f + n_l$. 

Each node $v \in V$ has a scalar valued state $x_v$ that represents its opinion.
Each follower node $v \in F$ executes a continuous version of the French-DeGroot opinion dynamics,
\begin{align*}
x_v(0)  &= x_v^{0} \\
\dot{x_v}(t)  &= - \sum_{u \in \Ne_v} (x_v(t)- x_u(t))
\end{align*}
where $x_v^{0}$ is the initial opinion of node $v$ and  $\Ne_i$ denotes the neighbor set of node $v$.
Leader states are initialized to 0 (for $v \in S_0$) or 1 (for $v \in S_1$), and the leader states remained fixed through the execution of the algorithm, i.e.,
for $v \in (S_1 \cup S_1)$,
\[
\dot{x}_v(t) = 0.
\]

Let $L$ be the Laplacian matrix of the graph, $L=D-A$, where $D$ is the diagonal matrix of the node degrees and $A$ is the adjacency matrix. 
Without loss of generality, we partition the nodes into leaders and followers so that their states $\textbf{x}$ can be written as
\[\textbf{x}=[ \xl^T \: \xf^T]^T\]
where $\xl$ is the vector of leader states of length $n_l$ and $\xf$ is the vector of follower states of length $n_f$. Then $L$ can be written as a block matrix:
\[L= \begin{bmatrix}
L_{ll} & L_{lf} \\
L_{fl} & L_{ff}
\end{bmatrix},
\]
where $L_{ff}$ is an $n_f \times n_f$ matrix of the interactions between followers and $L_{fl}$ is an $n_f \times n_l$ matrix representing the influence the leader nodes have on the followers. 
The dynamics can be expressed more compactly as
\begin{align*}
\dot{\xl}(t) &= 0 \\
\dot{\xf}(t) &= -L_{ff}\xf - L_{fl} \xl.
\end{align*}


It has been shown that, under these dynamics, the followers' states converge to a convex combination of the leader states, which is given by the following expression~\cite{ME10}
\begin{align}
\hat{\textbf{x}}_f = -L_{ff}^{-1}L_{fl}\xl. \label{follower.eq}
\end{align}
We call $\hat{\textbf{x}}_f$ the \emph{opinion vector} of the network $G$.
We note that since all leader opinions are either $0$ or $1$, for each follower $v$, $\hat{\bf{x}}_{f_v} \in [0,1]$.

\subsection{Diversity Performance Measures}
We quantify the diversity of the opinion vector $\hat{\textbf{x}}_f$ using two different diversity measures. 
Our first measure is based on the Simpson Diversity Index. 
This index was originally introduced as a measure of biological diversity, where a region with an even distribution of species is considered to be more diverse than an area where the population is dominated by only a few types of organisms. 
\begin{definition}
Consider a region with $R$ species, where each species $i$ has $n_i$ members present in the area.
The \emph{Simpson Diversity Index} is \cite{S49}:  
\[SDI=1-\frac{\sum_{i=1}^{R} n_i(n_i-1)}{n(n-1)}.\]
\end{definition}
Under this definition, $SDI=1$ represents infinite diversity, and $SDI=0$ indicates complete domination by a single species or category. 

We adapt this measure to the opinion vector by first discretizing the interval $[0,1]$ into  $R$ \emph{bins}: 
\begin{align}
b_1 &= \left[0, \frac{1}{R}\right) \label{bin1.eq} \\
b_i &= \left[\frac{i-1}{R}, \frac{i}{R}\right), ~~~i = 2 \ldots (R-1) \label{bin2.eq}  \\
b_{R} &=  \left[\frac{R-1}{R}, 1\right]. \label{bin3.eq}
\end{align}
We then count the number of opinions in $\hat{\textbf{x}}_f$  that fall into each bin.
The opinion diversity is measured as follows.
\begin{definition}
Let $c_i$ be the number of components of  $\hat{\textbf{x}}_f$ that lie in bin $b_i$.  
The \emph{Simpson Opinion Diversity Index} of a network $G$ with 0-leader set $S_0$ and 1-leader set $S_1$ is:
 \begin{align}
\Simpson(S_0, S_1, R) =1-\frac{\sum_{i=1}^{R} c_i(c_i-1)}{n_f(n_f-1)}.
\end{align}
\end{definition}

Our second diversity measure is based on  the Shannon Index,  a measure that was developed to quantify the entropy in a text string.
%
This index has also been applied to measuring diversity in ecosystems.
\begin{definition}
Consider a region with $R$ species, where each species $i$ has $n_i$ members present in the area.
The \emph{Shannon Index} is \cite{S01}:
\[  SI=-\sum_{i=1}^R p_i \ln(p_i)\]
where $p_i$ is the number of individuals that belong to category $i$ divided by the number of categories. 
\end{definition}
Note that this index is maximized when the individuals are evenly distributed among all $R$ categories, and is always non-negative since $p_i \in [0,1]$.

To adapt this measure to the opinion vector, we follow a similar procedure as for the Simpson Opinion Diversity Index.
We divide the interval into $R$ bins as shown in (\ref{bin1.eq}) - (\ref{bin3.eq}) and use these bins in our definition.
\begin{definition}
Let $p_i$ be the proportion of components of $\hat{\textbf{x}}_f$  that lie in bin $b_i$, i.e., $p_i = c_i / R$, with $p_i \ln(p_i) = 0$ when  $p_i=0$.
The \emph{Shannon Opinion Diversity Index} of a network $G$ with 0-leader set $\sz$ and 1-leader set $\so$ is
\[
\Shannon(\sz, \so, R) = -\sum_{i=1}^{R} p_i \ln(p_i).
\]
\end{definition}

When there are $n_f$ buckets and there exist optimal leader sets $\sz^*$, $\so^*$ so that all $n_f$ follower opinions are uniformly distributed, then  $c_i=1$ for $i=1,\ldots, n_f$. 
It follows that 
\begin{align}
\Simpson(\sz^*, \so^*, n_f)=1 \label{maxRFSimp.eq}
\end{align}
and 
\begin{align}
\Shannon(\sz^*, \so^*, n_f)=-\ln{\left(\frac{1}{n_f}\right)},\label{maxRFShan.eq}
\end{align}
are the maximum values for both measures. Note, as $n_f$ increases, the maximum value of $\Shannon$ is unbounded.

When $R=2$ and the leader sets $\sz^*$, $\so^*$ are selected so that all $n_f$ followers are uniformly distributed, then $|c_1-c_2|\leq1$. Without loss of generality, let $c_1=\lfloor \frac{n_f}{2} \rfloor$ and $c_2=\lceil \frac{n_f}{2} \rceil$. Then 
\begin{align}
\Simpson(\sz^*, \so^*, n_f)=1-\frac{\lfloor \frac{n_f}{2} \rfloor(\lfloor \frac{n_f}{2} \rfloor-1)}{n_f(n_f-1)}- \frac{\lceil \frac{n_f}{2} \rceil(\lceil \frac{n_f}{2} \rceil-1)}{n_f(n_f-1)} \label{maxR2Simp.eq}
\end{align}
and 
\begin{align}
\Shannon(\sz^*, \so^*, n_f)=-\frac{\lfloor \frac{n_f}{2} \rfloor}{n_f}\ln{ \left(\frac{\lfloor \frac{n_f}{2} \rfloor}{n_f} \right)}-\frac{\lceil \frac{n_f}{2} \rceil}{n_f}\ln{ \left(\frac{\lceil \frac{n_f}{2} \rceil}{n_f} \right)} \label{maxR2Shan.eq}
\end{align}
are the maximum values for both measures.

In some networks, even when  $\sz^*$, $\so^*$ are chosen as the leader sets, uniform distribution of follower opinions is unachievable. In such a case, the above expressions serve as an upper bound on the resulting diversity indices.

\subsection{Problem Formulation}

Let the pre-existing 0-leader be denoted by $\szf$.
We want to determine where to place the 1-leader, $l_1$,  such that the the diversity of opinions present in $\hat{\textbf{x}}_f$ is maximized. 
We pose two optimization problems, one for each diversity measure.

The \emph{Leader selection problem for Simpson Opinion Diversity} is: 
\begin{equation} 
\begin{array}{ll}
\underset{l_!}{\text{maximize}} & \Simpson(\szf, l_1, R) \\
\end{array} \tag{LS1} \label{simpProb.eq}
\end{equation}

The \emph{Leader selection problem for Shannon Opinion Diversity} is:
\begin{equation}
\begin{array}{ll}
\underset{l_1}{\text{maximize}} & \Shannon(\szf, l_1, R) \\
\end{array} \tag{LS2} \label{shanProb.eq}
\end{equation}


\section{Analysis} \label{results.sec}
We now present analytical solutions to (\ref{simpProb.eq}) and (\ref{shanProb.eq}) for $K=1$ in several classes of graphs for $R=n_f$ and $R=2$.
Since $K=1$, we call the single 0-leader $l_0$ and the 1-leader $l_1$.

\subsection{Path Graphs}
We first consider a path graph of  $n$ nodes, numbered $1, 2, \ldots, n$.  The following theorem addresses the  case where $R=n_f$ for both diversity indices and proves that the optimal placement of $l_1$ is at the farthest node from $l_0$.
\begin{theorem} \label{path1.thm}
Consider a path of length $n$, with a single 0-leader node $k$. The optimal solution to both (\ref{simpProb.eq}) and (\ref{shanProb.eq}), for $R = n_f$,
is to select node $j$ as the single 1-leader node, where $j=n$ if $k < n/2$ and $j=1$ otherwise.
\end{theorem}
%

To prove this theorem, first we note the useful fact that, for a path graph with $l_0 = k$ and $l_1=j$, where $k<j$~\cite{VFFO14}, 
\begin{align} {\hat{\bf{x}}_{f_v}=\frac{v-k}{j-k}} \text{ for }v=k+1, \ldots, j-1. \label{fact1.eq} \end{align} 
We also make use of the following lemmas.

\begin{lemma} \label{bins.lem}
Consider a path of $n$ nodes with $l_0=1$ and $l_1=n$, so that $\hat{\bf{x}}_{f_i}=\frac{i}{n_f+1}$ for $i=1, \ldots, n_f$. Then for all bins, $c_i=1$.
\end{lemma}
\begin{proof}
For all $i=1, \ldots n_f$, bin $b_i$ is bounded by $\frac{i-1}{n_f}$ below and $\frac{i}{n_f} $ above, and $\hat{\bf{x}}_{f_{i}}=\frac{i}{n_f+1}$. Observe that
${\frac{i-1}{n_f}-\frac{i}{n_f+1} < 0}$ and  ${\frac{i}{n_f}-\frac{i}{n_f+1} >0}$.
Thus $\hat{\bf{x}}_{f_i}$ falls in bin $b_i$  for all $i=1, \ldots n_f$.
\end{proof}
\begin{lemma} \label{bins2.lem}
Consider a path of $n$ nodes with 0-leader node $1$ and 1-leader $n-1$ so that $\hat{\bf{x}}_{f_i}=\frac{i}{n_f}$ for $i=1, \ldots, n_f-1$ and $\hat{\bf{x}}_{f_{n_f}}=1$. Then $c_1=0$, $c_{n_f}=2$, and $c_i=1$ for $i=2,\ldots, n_f-1$.
\end{lemma}
\begin{proof}
The smallest follower state is $\hat{\bf{x}}_{f_1}=\frac{1}{n_f}$; therefore $\hat{\bf{x}}_{f_1}$ falls in bin $b_2$. Note that the difference between follower states for followers $1$ to $n_f-1$ is $\frac{1}{n_f}$, therefore each follower $i$ lies in bin $i+1$ for $i=2, \ldots, n_f-1$  while $\hat{\bf{x}}_{f_{n_f}}=1$. Therefore, bins $i=2, \ldots, n_f-1$ all have count $c_i=1$, while bin $b_{n_f}$ has count $c_{n_f}=2$.
  \end{proof}
We now give the proof of Theorem~\ref{path1.thm}.

\begin{proof}
Without loss of generality, assume $\l_0 = k$ with $k <\frac{n}{2}$. Further,  assume
 $l_1 = j$, with $j>k$.
We can then note that there are $k-1$ nodes with opinion $0$, and $n-j$ nodes with opinion $1$.
Thus, $c_1\geq k-1$ and $c_{n_f}\geq n-j$. Note that all bins have width $\frac{1}{n_f}$. When the number of followers between $l_0$ and $l_1$ is $z<n_f$, the distance between the opinions of each consecutive pair of followers is $\frac{1}{z}<\frac{1}{n_f}$ by (\ref{fact1.eq}). Therefore, each bin $b_i$, $i=2, \ldots, n_f-1$, must have $c_i\leq1$.
Consider two cases:
\paragraph*{Case 1}Let $j-k=n_f$, so that $k=1$ and $j=n$. Then by Lemma \ref{bins.lem}, $c_i=1$ for all $i=1, \ldots, n_f$ and ${\Simpson(1, n, n_f) = 1}$.
When $l_1$ is node $n-1$, then by Lemma \ref{bins2.lem}, $c_1=0$, $c_{n_f}=2$, and $c_j=1$ for all $j=2, \ldots, n_f-1$. Then
\[\Simpson(1, n-1, n_f) = 1- \frac{2}{n_f(n_f-1)}.\]
Clearly, $\Simpson(1, n, n_f)-\Simpson(1, n-1, n_f) >0$, and the follower states are more diverse when $l_1$ is $n$ than when $l_1$ is $n-1$.

\paragraph*{Case 2} Let $j<n$ so that $j-k<n_f$. Then $c_1=k-1$ and $c_{n_f}=g$ where $g\geq n-j$. When $l_1$ is node $j-1$, the difference  between the opinions of follower nodes $i$ and $i+1$, where $i=2, \ldots, j-2$, increases from $\frac{1}{k-j}$ to $\frac{1}{k-j-1}$ by (\ref{fact1.eq}). Therefore, bins $i=2, \ldots, n_f-1$ have count $c_i\leq1$. The count of bin $b_{n_f}$ increases to $g+1$, since node $j$ now has opinion $1$ as well, while the count of bin $b_1$ is unchanged. 

Then $\Simpson$ is computed as follows:
\[\Simpson(1, j, n_f) = 1-\frac{(k-1)(k-2)}{n_f(n_f-1)} - \frac{g(g-1)}{n_f(n_f)}\]
and 
\[\Simpson(1, j-1, n_f) = 1- \frac{(k-1)(k-2)}{n_f(n_f-1)} - \frac{(g+1)g}{n_f(n_f)}.\]
The difference is then 
\begin{align*}
&\Simpson(1, j, n_f)-\Simpson(1, j-1, n_f) \\
&=\frac{(g+1)g}{n_f(n_f)}-\frac{g(g-1)}{n_f(n_f)}\\
&=\frac{2g}{n_f(n_f-1)}>0;
\end{align*}
therefore, moving $l_1$ from $j$ to $j-1$ decreases $\Simpson$.

We use these same two cases and their associated bin counts for $\Shannon$. In Case 1:
\[ \Shannon(1, n, n_f)=-n_f \frac{1}{n_f}\ln{\left(\frac{1}{n_f}\right)}\]
 and
\[ \Shannon(1, n-1, n_f)=- (n_f-2)\frac{1}{n_f}\ln{\left(\frac{1}{n_f}\right)} - \frac{2}{n_f}\ln{\left(\frac{2}{n_f}\right)}\]
and their difference is:
\begin{align*}
\Shannon(1, n, n_f) - \Shannon(1, n-1, n_f) =\frac{2}{n_f}\ln{2}>0.
\end{align*}

In Case 2:
\begin{align*}
& \Shannon(k, j, n_f)=-\frac{k-1}{n_f}\ln{\left(\frac{k-1}{n_f}\right)}\\
 &~~~~-\frac{g}{n_f}\ln{\left(\frac{g}{n_f}\right)}-(n_f-g-k+1)\frac{1}{n_f}\ln{\left(\frac{1}{n_f}\right)}
 \end{align*}
 and
\begin{align*}
& \Shannon(k, j-1, n_f)=-\frac{k-1}{n_f}\ln{\left(\frac{k-1}{n_f}\right)}\\
 &~~~~-\frac{g+1}{n_f}\ln{\left(\frac{g+1}{n_f}\right)}-(n_f-g-k)\frac{1}{n_f}\ln{\left(\frac{1}{n_f}\right)}
 \end{align*}and their difference is:
\begin{align*}
&\Shannon(k, j, n_f) - \Shannon(k, j-1, n_f) \\
&=  \frac{g}{n_f}\ln{\left(\frac{g+1}{g}\right)}+ \frac{1}{n_f}\ln{(g+1)}>0.
\end{align*}

We can see that in both cases, moving $l_1$ one node closer to $l_0$ decreases $\Shannon$. 
Thus, for both problems and both cases, moving $l_1$ closer to $l_0$ always decreases the diversity, and therefore, placing $l_1$ at node $n$ is optimal. 

A similar argument can be used to show that it is suboptimal to select $j<k$ for $l_1$ for both $\Simpson$ and $\Shannon$. 
\end{proof}

By Theorem \ref{path1.thm}, when $l_0  = k < n/2$, the optimal 1-leader is $l_1=n$ for both diversity indices. The resulting diversity measures are:
\begin{align*}
\Simpson(k, n, n_f)&=\textstyle 1-\frac{(k-1)(k-2)}{n_f(n_f-1)} \\
\Shannon(k,n,n_f) &=-\textstyle \frac{k-1}{n_f}\ln{\left( \frac{k-1}{n_f}\right)}-\frac{n_f-(k-1)}{n_f}\ln{\left( \frac{1}{n_f}\right)}.
\end{align*}
Thus, even when $l_1$ is chosen optimally, the diversity may be far from the maximal diversity value of $1$ or $-\ln{\left(\frac{1}{n_f}\right)}$, for (\ref{simpProb.eq}) and (\ref{shanProb.eq}), respectively.
And, these indices are farthest from their maximal value when $l_0=(n/2) - 1$.

%

We now consider the maximization of Problems (\ref{simpProb.eq}) and (\ref{shanProb.eq}) when $R=2$ and prove that, unlike when $R=N_f$, diversity is maximized when $l_1$ is chosen so that $l_0$ and $l_1$ are the same distance from the endpoints.
\begin{theorem} 
Consider a path of $n$ nodes, with a single 0-leader $l_0=k$. An optimal solution to both (\ref{simpProb.eq}) and (\ref{shanProb.eq}), for $R = 2$, is to select node $j$ as the single 1-leader, where $j=n-k+1$ when $k<\frac{n}{2}$ and $j=n-k$ otherwise.
\end{theorem}
\begin{proof}
Note that for $R=2$, the indices are computed simply as:
\begin{align*}
\Simpson(S_0, S_1, B_2)=1-\frac{c_1(c_1-1)}{n_f(n_f-1)} - \frac{c_2(c_2-1)}{n_f(n_f-1)} 
\end{align*}
and 
\begin{align*}
\Simpson(S_0, S_1, B_2) = -\frac{c_1}{n_f}\ln{\left(\frac{c_1}{n_f}\right)} -\frac{c_2}{n_f}\ln{\left(\frac{c_2}{n_f}\right)},
\end{align*}
 and that, by definition, both $\Simpson$ and $\Shannon$ are maximized when the opinions of the $n_f$ followers are evenly distributed between the two bins. 

Without loss of generality, assume $k<\frac{n}{2}$.
There are two cases:
\paragraph*{Case 1} Let $n_f$ be even.  Observe that when $l_1 = j=n-k-1$, there are $n_f-2(k-1)$ follower nodes that fall between $l_0$ and $l_1$. Then $c_1=k-1+\frac{n_f-2(k-1)}{2}=\frac{n_f-2}{2}=c_2$.  

\paragraph*{Case 2} Let $n_f$ be odd. Then it is impossible for $c_1$ and $c_2$ to be equal. Let $l_1$ be node $j=n-k+1$. Then once again there are $n_f-2(k-1)$ follower nodes between $l_0$ and $l_1$, and so $c_1=k-1+\frac{n_f-2(k-1)-1}{2} = \frac{n_f-1}{2}$ and $c_2=k-1+\frac{n_f-2(k-1)+1}{2} = \frac{n_f+1}{2}$ and thus $c_1=c_2-1$.

In both cases, both diversity indices are maximized when $l_1$ is node $j=n-k+1$. 

A similar argument can be used to show that both diversity indices are maximized when $l_1$ is node $n-k$ when $k>\frac{n}{2}$.
\end{proof}

Unlike the case when $R=n_f$, when $R=2$ and $n_f$ is even, regardless of which node is $l_0$, it is always possible to find an $l_1$ such the resulting opinion diversity is maximal, as given in (\ref{maxR2Simp.eq}) and (\ref{maxR2Shan.eq}).


\subsection{Cycle graphs}

We now consider (LS1) and (LS2), when $R=n_f$, over a cycle of $n$ nodes, numbered $1, 2, \ldots, n$, in a clockwise manner.   
We assume, without loss of generality, that $l_0$ is node $1$.
In such a setting, we prove that diversity is maximized when $l_1$ is placed directly beside $l_0$. 
\begin{theorem} \label{cycle1.thm}
Consider a cycle of  $n$ nodes with 0-leader node $l_0 = 1$, and let $R=n_f$ The optimal solutions to Problems (\ref{simpProb.eq}) and (\ref{shanProb.eq}) are $l_1=2$ and $l_1=n-1$, respectively.
\end{theorem}
\begin{proof}
Note that when $l_1$ is either node $2$ or $n-1$, the follower states $\hat{\textbf{x}}_f$ are the same as when the graph is a path of length $n$ with $l_0$ and $l_1$ located at the endpoints. Thus, by the same argument as in the proof of Theorem \ref{path1.thm}, $c_{i}=1$ for all $i=1, \ldots n_f$. When $l_1$ is at node $u$, $u \not = 2, \:n-1$, the cycle is broken into two paths, with one path having $p$ nodes between $l_0$ and $l_1$ and the other having $n_f-p$ nodes (without loss of generality, assume $p \leq n_f-p$).  Since both resulting paths have a length less than $n$, by (\ref{fact1.eq}) we note that $c_1=0$ for both paths and $c_i\geq 2$ for some $i \not = 1$. Thus, 
\[
\Simpson(1,2,n_f)=\Simpson(1,n-1,n_f)>\Simpson(1,u,n_f)
\]
and 
\[
\Shannon(1, 2, n_f)=\Shannon(1,n-1,n_f)>\Shannon(1,u,n_f).\]
 Therefore, the optimal $l_1$'s are $2$ and $n-1$.
\end{proof}

 When $R=n_f$, $l_0=1$ and an optimal 1-leader $l_1=2$ or $l_1=n-1$ is selected, then 
\[\Simpson(1, l_1, n_f)=1\]
and 
\[\Shannon(1, l_1, n_f)=- \ln{ \left(\frac{1}{n_f}\right)} \]
Clearly, $\Simpson(1, l_1, n_f) = \Simpson(\sz^*, \so^*, n_f)$ and $\Shannon(1, l_1, n_f)=\Shannon(\sz^*, \so^*, n_f)$. 
Therefore, when $R=n_f$ and the graph is a cycle, the maximal  diversity can be achieved, regardless of the location of $l_0$.

Next, we present results for the case where  $R=2$.
\begin{theorem} \label{R2cycle.thm}
Consider a cycle of  $n$ nodes with 0-leader node $l_0=1$. When $R=2$,  and $n_f$ is odd, $l_1=j$ is an optimal solution for all $j=2, \ldots, n$ for both Problems (\ref{simpProb.eq}) and (\ref{shanProb.eq}). When $n_f$ is even, $l_1=j$ is an optimal solution for all $j=2, 4, \ldots, n$ for both (\ref{simpProb.eq}) and (\ref{shanProb.eq}).
 \end{theorem}
\begin{proof}
Assume that $l_1=j$, such that the cycle is broken into two paths. Let one path have $p\geq 0$ nodes between $l_0$ and $l_1$ and the other have $n_f-p$. Without loss of generality, assume $p \leq n_f-p$.

We once again note that, in this setting,
\begin{align*}
\Simpson(S_0, S_1, B_2)=1-\frac{c_1(c_1-1)}{n_f(n_f-1)} - \frac{c_2(c_2-1)}{n_f(n_f-1)} 
\end{align*}
and 
\begin{align*}
\Simpson(S_0, S_1, B_2) = -\frac{c_1}{n_f}\ln{\left(\frac{c_1}{n_f}\right)} -\frac{c_2}{n_f}\ln{\left(\frac{c_2}{n_f}\right)}.
\end{align*}
By definition, both $\Simpson$ and $\Shannon$ are maximized when $n_f$ is evenly distributed between the two bins. 
We consider four cases and show that in each the location of $l_1$ has no effect on the bin counts $c_1$ and $c_2$.

\paragraph*{Case 1} Let $n_f$ be odd and $p$ be odd. Then $n_f-p$ must be even. Then $c_1=\frac{p-1}{2}+\frac{n_f-p}{2} = \frac{n_f-1}{2}$ and $c_2=\frac{p+1}{2}+\frac{n_f-p}{2}=\frac{n_f+1}{p}$.
\paragraph*{Case 2} Let $n_f$ be odd and $p$ be even so that $n_f-p$ must be odd. Then $c_1=\frac{p}{2}+\frac{n_f-p-1}{2} = \frac{n_f-1}{2}$ and $c_2=\frac{p}{2}+\frac{n_f-p+1}{2}=\frac{n_f+1}{p}$.
\paragraph*{Case 3} Let $n_f$ be even and $p$ be odd, so $n_f-p$ is  also odd. Then $c_1=\frac{p-1}{2}+\frac{n_f-p-1}{2} = \frac{n_f-2}{2}$ and $c_2=\frac{p+1}{2}+\frac{n_f-p+1}{2}=\frac{n_f+2}{p}$.
\paragraph*{Case 4} Finally, let $n_f$ be even, so that both $p$ and $n_f-p$ are even. Then $c_1=\frac{p}{2}+\frac{n_f-p}{2} = \frac{n_f}{2}$ and $c_2=\frac{p}{2}+\frac{n_f-p}{2}=\frac{n_f}{p}$.

In all four cases, $c_1$ and $c_2$ are independent of $p$ and $n_f-p$. When $n_f$ is odd, all possible locations of $l_1$ are equivalent and have $|c_1-c_2|=1$, and, thus, all $l_1=j$, $j=2, \ldots, n$ optimize both performance measures. When $n_f$ is even, ${|c_1-c_2|\leq 1}$ only when $p$ and $n_f-p$ are both even. Therefore $l_1=j$ is equivalent for all $j=2, 4, \ldots, n$, and all such $l_1=j$ optimize both performance measures.
\end{proof}

 When $R=2$, the graph is a cycle, $l_0=1$, and $n_f$ is odd, then all possible placements of $l_1$ are optimal.  In the case where $n_f$ is even, $l_1=j$ is optimal only when $j$ is even. 

\subsection{Tree graphs}
Finally, we consider (\ref{simpProb.eq}) and (\ref{shanProb.eq}) over tree graphs. We first study the problem when $R=n_f$
in a special class of tree graph.
\begin{theorem} \label{simpleTree.thm}
Consider a tree graph of size $n$ where exactly one node $t$ has degree $deg(t)=3$ and all other nodes $i$ have $deg(i)\leq 2$. Let the 0-leader $l_0$ be at a leaf node. Then Problems (\ref{simpProb.eq}) and (\ref{shanProb.eq}) are both maximized when $l_1$ is at the end of the longest path from $l_0$.
\end{theorem}
The proof of this theorem is quite lengthy and is deferred to a technical report \cite{MP18}.

We also consider the problem of identifying an optimal $l_1$ in a general tree graph when $R=2$. 

Note that the set of followers in any graph with one 0-leader and one 1-leader can be partitioned into three sets, $P_1$, $P_2$, and $P_3$, based on the locations of the leaders. Let $path(a,b)$ be defined on a tree graph as the set of all nodes that lie on the path between nodes $a$ and $b$, inclusive of $a$ and $b$. Then, the three sets are defined as follows: 
\begin{align*}
P_1&=\{i \in V: l_0 \in path(i,l_1)\}, \\
P_2&=\{i \in V: l_1 \not \in path(i,l_0) \text{ and } l_0 \not \in path(i, l_1)\}, \\
P_3&=\{ i \in V : l_1 \in path(i, l_0)\}, 
\end{align*}
where 
$|P_1 \cup P_2 \cup P_3|=n_f$
and  
$P_1 \cap P_2 = P_1 \cap P_3 = P_2 \cap P_3 = \emptyset.$
Note also that, in a tree, the \emph{graph distance} between two nodes, $d(u,v)$, is the length of the path between nodes $u$ and $v$.
In graph $G$, shown in Figure \ref{opttree.fig}, there are $10$ follower nodes. Given $l_0$, $l_1$ as shown, $P_1 =\{1, 2, 3\}$, $P_2=\{4,5,6,7\}$, and $P_3=\{8,9,10\}$.

We now consider the case when $R=2$ and $l_0$ and $l_1$ are placed such $|P_1|=|P_3|$, that is, the number of followers with opinion $0$ and opinion $1$ are the same.
If the network is such that the nodes that lie between the leaders have evenly distributed opinions, then the existing leader placement is optimal, regardless of the size of $|P_1|$ and $P_3$. An example network showing such a leader placement is given in Figure \ref{opttree.fig}.

\begin{figure}
\centering
\begin{tikzpicture}[scale=0.8]

\node[shape=circle,draw=black] (1) at (0,0.5) {$1$};
\node[shape=circle,draw=black] (2) at (0,-0.5) {$2$};
\node[shape=circle,draw=black] (3) at (1,0) {$3$};
\node[shape=circle,draw=black] (4) at (2,0) {$l_0$};
\node[shape=circle,draw=black] (5) at (3,0) {$4$};
\node[shape=circle,draw=black] (6) at (4,0) {$5$};
\node[shape=circle,draw=black] (7) at (5,0) {$6$};
\node[shape=circle,draw=black] (12) at (6,0) {$l_1$};
\node[shape=circle,draw=black] (8) at (7,0) {$8$};
\node[shape=circle,draw=black] (9) at (8,0.5) {$9$};
\node[shape=circle,draw=black] (10) at (8,-0.5) {$10$};
\node[shape=circle,draw=black] (11) at (3,1) {$7$};

\foreach \from/\to in {1/3, 2/3, 3/4, 4/5, 5/6, 6/7, 7/12, 12/8, 8/9, 8/10, 5/11}
    \draw (\from) -- (\to);

\end{tikzpicture}
\caption{Example of a tree network $G$, where, when $R=2$, $l_0$ and $l_1$ are optimal leaders, satisfying Theorem \ref{2tree.thm}.}
\label{opttree.fig}
\end{figure}
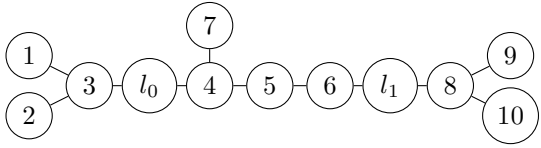

\begin{theorem} \label{2tree.thm}
Consider a tree graph $G$ with $n$ nodes, where $l_0=i$ and $|P_1|=k-1$, and let $R=2$. Without loss of generality, let $k<\frac{n}{2}$, so that $|P_2 \cup P_3| = n_f-(k-1)$.  If there exists a node $j$ such that, when $l_1=j$,  $|P_3|=k-1$, $|P_2|=n_f-2(k-1)$, and the opinions of nodes $i \in P_2$ are evenly distributed between bins $b_1$ and $b_2$, then $j$ is an $l_1$ that maximizes both (\ref{simpProb.eq}) and (\ref{shanProb.eq}).
\end{theorem}

\begin{proof}
We once again note that, when $R=2$, that
\begin{align*}
\Simpson(\szf, j, B_2)=1-\frac{c_1(c_1-1)}{n_f(n_f-1)} - \frac{c_2(c_2-1)}{n_f(n_f-1)}
\end{align*}
and 
\begin{align*}
\Shannon(\szf, j, B_2) = -\frac{c_1}{n_f}\ln{\left(\frac{c_1}{n_f}\right)} - -\frac{c_2}{n_f}\ln{\left(\frac{c_2}{n_f}\right)}.
\end{align*}
By definition, both $\Simpson$ and $\Shannon$ are maximized when $n_f$ is evenly distributed between the two bins so that $|c_1-c_2|\leq 1$.


There are two cases to consider.
\paragraph*{Case 1} When $n_f$ is even, both (\ref{simpProb.eq}) and (\ref{shanProb.eq}) are maximized when $c_1=c_2=\frac{n_f}{2}$. Let $l_1=j$ such that $|P_3|=k-1$ and $|P_2|=n_f-2(k-1)$. To ensure that the opinions of nodes $i\in P_2$ are divided evenly among $b_1$ and $b_2$, there must exist an edge $(a,b)\in E$, $a,b \in P_2$ with the following properties: $b \in path(a,l_1)$, $a \in path(b,l_0)$, $d(l_0, a)=d(b, l_1)$, and 
\begin{align*}
 |\{i: b \in path(i, l_1)\}|&=|\{j:a \in path(j, l_0)\}| \\
&= \frac{n_f-2(k-1)}{2}. 
\end{align*}
 Then $c_1=c_2= \frac{n_f}{2}$, and $j$ is an optimal 1-leader.
\paragraph*{Case 2} When $n_f$ is odd, both (\ref{simpProb.eq}) and (\ref{shanProb.eq}) are maximized when either
\[ c_1=\frac{n_f-1}{2}, ~ c_2=\frac{n_f+1}{2}\]
or
 \[c_1=\frac{n_f+1}{2},~c_2=\frac{n_f-1}{2},\] 
 as in both cases $|c_1-c_2|\leq 1$. 
  Let $l_1=j$ such that $|P_3|=k-1$ and $|P_2|=n_f-2(k-1)$. To ensure that the opinions of nodes $i\in P_2$ are divided evenly among $b_1$ and $b_2$, there must exist a node $v\in P_2$ with the following properties: 
  \begin{align*}
  &d(l_0,v)=d(v,l_1), \\
  &|\{i: v \in path(i, l_1)\} \setminus \{v\}|=g_1, \\
  &|\{j: v \in path(j, l_0)\}\setminus \{v\}|=g_2.
  \end{align*} 
  Note that, because $|P_2|$ is odd and $v$ is equidistant from $l_0$ and $l_1$, $v \in b_2$.
  Then the opinion diversity is maximized when either
  \[g_1=g_2=\frac{n_f-2(k-1)-1}{2}\]
   and, thus, $c_1=\frac{n_f-1}{2}$ and $c_2=\frac{n_f+1}{2}$, or 
   \begin{align*}
   g_1&=\frac{n_f-2(k-1)+1}{2} \text{ and }g_2=\frac{n_f-2(k-1)-3}{2},
 \end{align*}
   in which case $c_1=\frac{n_f+1}{2}$ and $c_2= \frac{n_f-1}{2}$. 
   For both pairs of $g_1$ and $g_2$, $|c_1-c_2|=1$ and $j$ is an optimal $l_1$.
   
   Therefore, in both cases, when $l_1=j$, $|c_1-c_2|\leq 1$ and $l_1$ maximizes both  (\ref{simpProb.eq}) and (\ref{shanProb.eq}) for $R=2$.
   \end{proof}
Note that depending on the structure of $G$, such a partition of the follower nodes may not be possible. However, this does not imply that there are no other optimal choices of $l_1$. 

\section{Numerical Examples} \label{example.sec}
We highlight some of the analysis in Section \ref{results.sec} via numerical examples, using the 
graph  $G$ shown in Fig. \ref{tree.fig}.

As shown in the previous section, for cycles and paths, the same $l_1$ node is optimal for both  (\ref{simpProb.eq}) and (\ref{shanProb.eq}) for both $R=n_f$ and $R=2$.
We now show that, when $R=n_f$, this relationship between the optimal solutions does not always hold. 

In the network $G$ in Fig. \ref{tree.fig}, when $l_0=1$, the optimal $l_1$ depends on the performance measure used. For (\ref{simpProb.eq}), the optimal 1-leader can be either $10$ or $11$, but for (\ref{shanProb.eq}) the optimal 1-leader is either $5$ or $6$. 

\begin{table} 
  \caption{Diversity indices for different $l_1$ values.}
\centering
  \begin{tabular}{ | c | c | c |}
    \hline
   $l_1$ & $\Simpson(1, l_1, n_f)$ & $\Shannon(1, l_1, n_f)$\\ \hline \hline
   2 & 0 & 0 \\ \hline
   3 & 0.5 & 0.637 \\ \hline
   4 & 0.556& 0.849 \\\hline
   5 & 0.583 & 1.003 \\ \hline
   6 & 0.583 & 1.003  \\ \hline
   7 & 0.556 & 0.687\\ \hline
   8 & 0.556 & 0.687 \\ \hline
   9 & 0.556 & 0.687 \\ \hline
   10 & 0.639 & 0.937\\ \hline
   11 & 0.639 & 0.937 \\ \hline
  \end{tabular} \label{div.table}
\end{table}

As shown in Table \ref{div.table}, we can see that, for $y\in\{10,11\}$, $z\in\{5,6\}$, 
\begin{align*}
&\Simpson(1,y, n_f)=0.639 \\
&~~~~~>\Simpson(1, z, n_f)=0.583 
\end{align*}
and
\begin{align*}
&\Shannon(1,z,n_f)=1.003 \\
&~~~~~> \Shannon(1, y, n_f)=0.937.
\end{align*}

We can use the same graph $G$ to observe that Theorem~\ref{simpleTree.thm} does not generalize to trees with more than one node $i$ with degree $deg(i)\geq3$. The longest path from $l_0$ in $G$ terminates at node $6$, but, as shown above, $6$ is not the optimal $l_1$ for Problem (\ref{simpProb.eq}).

Although determining the optimal $l_1$ is often simple, it is not a trivial problem. When $R=n_f$, $11$ and $6$ are optimal $l_1$ nodes in $G$ for  (\ref{simpProb.eq}) and (\ref{shanProb.eq}), respectively. The worst case $l_1$ node for both problems is $2$, which results in $\Simpson(1,2,n_f)=0$ and $\Shannon(1,2,n_f)=0$.

Finally, we note that node $11$ is an optimal $l_1$ node that satisfies the requirements listed in Theorem~\ref{2tree.thm}, such that $P_1=P_3=\emptyset$, $P_2=\{i: 2\leq i \leq 10\}$, and $c_1=\frac{n_f+1}{2}=5$ and $c_2=\frac{n_f-1}{2}=4$. When $l_1=8, ~9, ~ \text{or }10$, $c_1=5$ and $c_2=4$, but $|P_1|\not=|P_3|$. We can then see that there are optimal 1-leaders not described by Theorem~\ref{2tree.thm}.

\begin{figure}
\centering
\begin{tikzpicture}[scale=0.75]

\node[shape=circle,draw=black, fill=blue!20] (1) at (0,0) {$1$};
\node[shape=circle,draw=black] (2) at (1,0) {$2$};
\node[shape=circle,draw=black] (3) at (2,0) {$3$};
\node[shape=circle,draw=black] (4) at (3,0) {$4$};
\node[shape=circle,draw=black, fill=green!20] (5) at (4,0) {$5$};
\node[shape=circle,draw=black] (6) at (5,0) {$6$};
\node[shape=circle,draw=black] (7) at (1,-1) {$7$};
\node[shape=circle,draw=black] (8) at (0,-2) {$8$};
\node[shape=circle,draw=black] (9) at (2,-2) {$9$};
\node[shape=circle,draw=black, fill=red!20] (10) at (1,-2) {$10$};
\node[shape=circle,draw=black] (11) at (1,-3.2) {$11$};

\foreach \from/\to in {1/2, 2/3, 3/4, 4/5, 5/6, 2/7, 7/8, 7/9, 7/10, 10/11}
    \draw (\from) -- (\to);

\end{tikzpicture}
\caption{Example of a tree network $G$, where one optimal leader for $\Simpson$ is shown in red, and one optimal leader for $\Shannon$ is shown in green.}
\label{tree.fig}
\end{figure}

\section{Conclusion} \label{concl.sec}
We have proposed two diversity measures, adapted from ecology, for networks with French-DeGroot opinion dynamics, where the networks contain leaders with binary opposing opinions. Further, using these measures, we have formalized the problem of maximizing opinion diversity in a network that contains  a single leader with opinion 0 by selecting which node should become the leader with opinion 1. We have presented analytical solutions to these problems for the case of a single 0-leader and a single 1-leader in paths, cycles, and tree graphs. 
%
 In future work, we plan to study the problem of  optimal leader placement in more general graphs and with multiple leaders of both opinion types.

\bibliographystyle{IEEEtran}
\bibliography{opinDynBib}

\section{Appendix}
\subsection{Proof of Theorem \ref{simpleTree.thm}}
Before proceeding with the proof, we require the following lemmas:

\begin{lemma}[\cite{KR93} Lem. E]
\label{cutpoint.lem}
Let graph $G=(V,E)$ be partitioned into two components, $A$ and $B$, that share only a single vertex $x$.
The resistance distance between any two vertices $u \in A$ and $v \in B$ is: 
\begin{equation*}
r(u,v) = r(u,x) + r(x,v).
\end{equation*}
\end{lemma}

\begin{figure}
\centering
\begin{tikzpicture}[scale=0.6]

\node[shape=circle,draw=black, minimum size=.75 cm] (1) at (0,0) {$l_0$};
\node[shape=circle,draw=black,minimum size=.75 cm] (4) at (3,0) {$t$};
\node[shape=circle,draw=black,minimum size=.75 cm] (5) at (5,0) {$v$};
\node[shape=circle,draw=black,minimum size=.75 cm] (6) at (7,0) {$l_1$};
\node[shape=circle,draw=black,minimum size=.75 cm] (8) at (3,-1.5) {$u$};
\node[shape=circle,draw=black,minimum size=.75 cm] (9) at (3,-3) {};

\foreach \from/\to in {5/6}
    \draw (\from) -- (\to);
\foreach \from/\to in {4/5, 4/8, 8/9, 1/4}
    \draw (\from) [dotted, thick] -- (\to);

\end{tikzpicture}
\caption{Illustration of placement of nodes for Lemma \ref{branch.lem}.}
\label{branch.fig}
\end{figure}
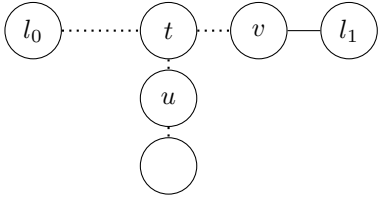

\begin{lemma} \label{branch.lem}
Consider a tree network with $n_f \geq 3$ where $l_0$ and $l_1$ are leaf nodes such that $deg(l_0)=deg(l_1)=1$ and at least two nodes lie between them. Consider a node $t$ such that $deg(t) \geq 3$ and $t \in path(l_0, l_1)$. Consider a node $u$, such that $t \in path(u, l_0)$ and $t \in path(u, l_1)$.  For a visual example, see Fig. \ref{branch.fig}. 
Then $\hat{\textbf{x}}_{f_u} = \hat{\textbf{x}}_{f_t}$ for all $u$ and $t$.
\end{lemma}
\begin{proof}
Let the node $v$ be such that $v \in path(l_0, l_1)$ and $(v,l_1) \in E$.
Note that, by inspection of (\ref{follower.eq}), for all nodes $i \in F$, $\hat{\textbf{x}}_{f_i}=\Lfi(i, v)$.
By the definition of resistance distance, we know that
\[ r(u,v)=\Lfi(u,u)+\Lfi(v,v)-2\Lfi(u,v).\]
We then rearrange this to 
\[ 2\Lfi(u,v) = \Lfi(u,u)+\Lfi(v,v) - r(u,v)\]
By applying Lemma \ref{cutpoint.lem} this then becomes 
\begin{align} 
2\Lfi(u,v) &= \Lfi(u,u)+\Lfi(v,v) - (r(u,t)+r(t,v)) \nonumber  \\
&= \Lfi(u,u)+\Lfi(v,v) - \Big(r(u,t) \nonumber \\
&~~~~+ \Lfi(t,t)+ \Lfi(v,v)-2\Lfi(t,v) \Big) \nonumber \\
&=  \Lfi(u,u)- r(u,t) - \Lfi(t,t)+2\Lfi(t,v). \label{resEq.eq}
\end{align}

By the same lemma, we note that 
\begin{align}
r(u,l_0 \cup l_1)=r(u,t)+r(t,l_0 \cup l_1). \label{cutpt.eq}
\end{align}
The resistance distance between node $u$ and the set of leaders is \cite{GBS08}
\[r(u, l_0 \cup l_1) = \Lfi(u,u) \]
 and the distance between node $t$ and the set of leaders is 
 \[
 r(t, l_0 \cup l_1)=\Lfi(t,t) .
 \] 
We can substitute the above facts into (\ref{cutpt.eq}) to get the equivalent expression
 \begin{align}
 \Lfi(u,u) - r(u,t) - \Lfi(t,t) = 0 \label{cutpt2.eq}
 \end{align} 
Then we can simplify (\ref{resEq.eq}) to find that
\[\Lfi(u,v)=\Lfi(t,v),\]
and therefore 
\[\hat{\textbf{x}}_{f_u}=\hat{\textbf{x}}_{f_v}\]
thus concluding the proof.
\end{proof}


We can now conclude with the proof of Theorem \ref{simpleTree.thm}.
\begin{proof}
By definition, there are only three leaf nodes in the tree, one of which has been pre-selected to be $l_0$. 
Let the remaining two leaf nodes of $G$ be $u_1$ and $u_2$ and let $d(t, u_1)=p_1$ and $d(t, u_2)=p_2$. Without loss of generality, let $p_1  \geq p_2$. Let $u_1'$ be the neighbor of $u_1$ and let $u_2'$ be the neighbor of $u_2$. Finally, recall that node $t$ is the only node with degree $3$. Without loss of generality, let $\hat{\textbf{x}}_{f_t}$ fall in bin $b_j$.

We first show that placing $l_1$ at nodes $u_1$ and $u_1'$ is equivalent. 

When $l_1=u_1$, $|path(l_0,l_1)| \leq n_f-1$. By Lemma \ref{bins2.lem}, $c_2=0$ and the difference between two adjacent followers' opinions  on the path between $l_0$ and $l_1$ is at least $\frac{1}{n-2}$, by (\ref{fact1.eq}). Then  $c_{n_f} \leq 1$, with equality only when $p_2=1$. By Lemma \ref{branch.lem},  $c_j=p_2+1$, and thus the remaining $n_f-p_2-1$ bins all have count $1$. 

When $l_1=u_1'$, $\hat{x}_{f_{u_1}}=1$ and the distance between two adjacent followers' opinions on the path between $l_0$ and $l_1$ is at least $\frac{1}{n-3}$. Therefore, only $\hat{x}_{f_{u_1}}$ falls in bin $b_{n_f}$ and hence the distribution of bin counts remains the same as when $l_1=u_1$.  Thus, $\Simpson(1, u_1, n_f)=\Simpson(1, u_1', n_f)$ and $\Shannon(1, u_1, n_f)=\Shannon(1, u_1', n_f)$.

A similar argument can be used to show the same result for placing $l_1$ at nodes $u_2$ and $u_2'$. Note that when $p_1=p_2$, $u_1$, $u_2$, $u_1'$, and $u_2'$ are all equivalent.

Now we show that $\Simpson(l_0, u_1, n_f) \geq \Simpson(l_0, u_2, n_f)$.
Recall that when $l_1=u_1$, $c_{j}=p_2+1$, and there are $n_f-p_2-1$ bins $i$ with count $b_i=1$.
Then 
\[\Simpson(l_0, u_1, n_f)=1-\frac{(p_2+1)p_2}{n_f(n_f-1)}.\]
Similarly, 
\[\Simpson(l_0, u_2, n_f)=1-\frac{(p_1+1)p_1}{n_f(n_f-1)}.\]
The difference is then 
\begin{align*}
&\Simpson(l_0, u_1, n_f) - \Simpson(l_0, u_2, n_f) \\
&~~=\frac{(p_1+1)p_1}{n_f(n_f-1)}  - \frac{(p_2+1)p_2}{n_f(n_f-1)},
\end{align*}
which by assumption is always non-negative.

Similarly,
\begin{align*}
\Shannon(l_0, u_1, n_f)&=-\frac{p_2+1}{n_f}\ln{ \left( \frac{p_2+1}{n_f} \right)} \\
&- (n_f-p_2-1) \frac{1}{n_f}\ln{ \left( \frac{1}{n_f} \right)} 
\end{align*}
and 
\begin{align*}
\Shannon(l_0, u_2, n_f)&=-\frac{p_1+1}{n_f}\ln{ \left( \frac{p_1+1}{n_f} \right)} \\
&- (n_f-p_1-1) \frac{1}{n_f}\ln{ \left( \frac{1}{n_f} \right)} .
\end{align*}
The difference is then
\begin{align}
&\Shannon(l_0, u_1, n_f)-\Shannon(l_0, u_2, n_f) \nonumber \\
~~&=\frac{p_1+1}{n_f}\ln{ \left( \frac{p_1+1}{n_f} \right)} - \frac{p_2+1}{n_f}\ln{ \left( \frac{p_2+1}{n_f} \right)} \nonumber \\
&-(p_1-p_2) \frac{1}{n_f}\ln{ \left( \frac{1}{n_f} \right)}. \label{shanDiffTree.eq} 
\end{align}

Let $p_1=p_2+k$, for some $k\geq 0$. Then (\ref{shanDiffTree.eq}) becomes:
\begin{align*}
&\Shannon(l_0, u_1, n_f)-\Shannon(l_0, u_2, n_f) \\
~~&=\frac{p_2+k+1}{n_f}\ln{ \left( \frac{p_2+k+1}{n_f} \right)} - \frac{p_2+1}{n_f}\ln{ \left( \frac{p_2+1}{n_f} \right)}\\
&~~~~-k \frac{1}{n_f}\ln{ \left( \frac{1}{n_f} \right)}  \\
~~&=\frac{p_2+1}{n_f}\ln{ \left( \frac{p_2+k+1}{n_f} \right)} - \frac{p_2+1}{n_f}\ln{ \left( \frac{p_2+1}{n_f} \right)}\\
&~~~~+ \frac{k}{n_f}\ln{ \left( \frac{p_2+k+1}{n_f} \right)} -\frac{k}{n_f}\ln{ \left( \frac{1}{n_f} \right)}  \\
~~&=\frac{p_2+1}{n_f}\ln{ \left( \frac{p_2+k+1}{p_2+1} \right)}+ \frac{k}{n_f}\ln{ \left(p_2+k+1\right)},
\end{align*}
which is non-negative for all such $k$. 

Therefore, choosing $l_1$ to be either $u_1$ or $u_1'$, when $u_1$ is the node farthest from $l_0$, is optimal for both (\ref{simpProb.eq}) and (\ref{shanProb.eq}).
\end{proof}

\end{document}